\documentclass[12pt]{article}
\usepackage[utf8]{inputenc}

\title{Iteration entropy} 
\author{Joachim von zur Gathen}
\date{\today}

\usepackage{pstricks}
\usepackage{natbib}
\usepackage{graphicx}
\usepackage{amsmath}
\usepackage{amsthm}
\usepackage{amsfonts}
\usepackage{caption}
\usepackage{subcaption}

\usepackage{pst-node}

\newtheorem{thm}{Theorem}

\newtheorem{defn}{Definition}

\numberwithin{equation}{section} \numberwithin{thm}{section}
\numberwithin{lem}{section}
\numberwithin{problem}{section}
\numberwithin{prop}{section}
\numberwithin{cor}{section}
\numberwithin{conj}{section}
\numberwithin{obs}{section}

\newcommand{\Z}{{\mathbb Z}}
\newcommand{\N}{{\mathbb N}}
\newcommand{\F}{{\mathbb F}}

\newtheorem{example}{Example}
\numberwithin{example}{section}

\newcommand{\entfk}{H_{f,k}}
\newcommand{\entinf}{H_{f,\infty}}

\begin{document}

\maketitle

\begin{abstract}
We apply a common measure of randomness, the entropy, in the context
of iterated functions on a finite set with $n$ elements.
For a permutation, this entropy turns out to be asymptotically
(for a growing number of iterations) close to $ \log_2 n$ minus the entropy
of the vector of its cycle lengths.
For general functions, a similar approximation holds.
%
\end{abstract}

\section{Introduction}
Arithmetic dynamics deals with discrete dynamical systems given by an (arithmetic)
function on a finite set.
Of particular interest are polynomials over a finite field or ring.
Their iterations form a well-studied subject with many applications.
In the area of cryptography, one is interested in showing some randomness properties
of such iterations.
The special case of power maps is discussed in Example \ref{powermap}.
Ideally, one would like to exhibit specific functions on finite sets whose iterations, beginning with a 
uniformly random starting value, provide uniformly random values, or at least that its values
form a pseudorandom sequence.
This goal seems out of reach at the present.

More modestly, one tries to show certain randomness properties of such a function
such as (approximate) equidistribution.
The \emph{functional graph} of $f$ has the base set as its nodes and a directed edge
from $x$ to $y$ if $f(x)=y$.
One may consider
certain graph parameters like the numbers and sizes of connected components or cycles
and ask whether they
are (approximately) distributed for the functions under consideration
 as they are for general functions.%

Beginning with \cite{flaodl90}, also \cite{flygar14, belgar16, brigar17}
studied functions and polynomials from this perspective.
For a uniformly random map on $n$ points, the expected size of the giant component
(an undirected component of largest size) in its functional graph is $\mu n$
with $\mu \approx 0.75788$ (\cite{flaodl90}, Theorem 8 (ii)).
Certain classes of polynomials over finite fields, mostly of small degree,
are considered in \cite{marpan16}.
\cite{konluc16} present theoretical and experimental results on maps given by random quadratic polynomials
 over a finite prime field. 
The expected size of the giant component coincides with that for random maps.
However, the number of cyclic points (points on a cycle in the functional graph)
is much smaller. In their experiments with the ten primes following 500\,000,
this is only about 885.
\cite{ostsha16} extend some of this to certain rational functions.

\cite{arrtav92} and \cite{bursch17} show precise results on the distribution of cycle lengths
for polynomials over a finite field, and also for rational functions.
The least common multiple $T$ of all cycle lengths is the order for a permutation and might
be called the asymptotic order for a general function.
They prove a lower bound $\frac d 2 (1+o(1))$ for $\log T$.
\cite{marpan17} consider the distribution of this value, and also the number of cyclic points, for 
special types of polynomials.

In uniformly random permutations, the expected length of a longest cycle is $\tau n$
with $\tau \approx 0.62433$;  \cite{shello66} give an exact expression for $\tau$,
and much statistical information about the cycle length of random permutations,
including the moments of the $r$th shortest and longest lengths, for $r=1, 2, \ldots$.
\cite{mansha17}  find the average number of cyclic points for quadratic polynomials
to be about the average size of a longest cycle, namely, close to $\sqrt{2n/\pi}$.
For $n=500\,000$, this evaluates to about 564.
Thus there is a substantial difference of the expected largest cycle lengths
between random permutations and general functions. 

In this paper, we take a different route.
We define a general notion of \emph{iteration entropy}, applicable to any function from a finite set
to itself.
For a growing number of iterations, it approaches a limit which forms the central concept of this paper,
the \emph{asymptotic iteration entropy}.
This measure abstracts from individual values like number or size of components or cycles
by including them in a single parameter.
It enjoys some natural properties like convexity for disjoint unions of functions.
One can compare different functions under this measure.
For example, when we fix the component sizes (summing to $n$),
then permutations have a larger asymptotic iteration entropy than other functions.

For the connected components of size $t$ contaning a cycle of size $c$, the values $t \log_2 c$
make up the asymptotic iteration entropy (up to a factor of $n$).
This suggests as an open question the study of this parameter,
or, more generally, the joint distribution of $(t,c)$ in functional graphs.

\section{The iteration entropy}
We let $X$ be a finite set with $n$ elements, $f\colon X \rightarrow X$ a map, and
for a nonnegative integer $j$, $f^{(j)} = f \circ f \circ \cdots \circ f$ (with $j$ copies of $f$)
its $j$th iteration. Thus $f^{(0)} = \text{id}$. For a positive integer $k$ and $x,y \in X$,
we denote as
\begin{equation*}
N_{f,k} (x,y) = \# \{ j \in \N \colon 0\leq j < k, f^{(j)}(x) = y \}
\end{equation*}
the number of times that $x$ is mapped to $y$ by an iterate of $f$, before the $k$th one.
Then
\begin{align*}
\sum_{x,y \in X} N_{f,k}(x,y) & = kn,\\
N_{f,k}(x,y) & \leq k \text{ for all } x, y,
\end{align*}
and
\begin{equation*}
\label{defp}
p_{f,k}(x,y) = \frac {N_{f,k}(x,y)} {kn}
\end{equation*}
defines a probability distribution on $X^2$, with all $p_{f,k}(x,y)$ at most $1/n$.
The usual Shannon entropy $H^*(p_{f,k})$ of this distribution is
$$
H^*(p_{f,k}) =  \sum_{x,y \in X}   {p_{f,k}(x,y)} \log_2 p^{-1}_{f,k}(x,y)
 = \sum_{x,y \in X}  \frac {N_{f,k}(x,y)} {kn}  \log_2 (\frac {kn} {N_{f,k}(x,y)} ).
$$
Throughout this paper, we employ
the usual convention that $z \log_2 z^{-1}$ is taken as $0$ when $z=0$.
The general upper bound on the entropy implies that $0 \leq H^*(p_{f,k}) \leq 2 \log_2 n$.

An observation by Igor Shparlinski leads to the following simplification:
we subtract $\log_2 n$ from this value.

\begin{defn}
\label{defItEnt}
The (shifted $k$th) {\em iteration entropy}  $H_{f,k}$ of $f$ is:
$$
H_{f,k} = H^*(p_{f,k}) - \log_2 n 
 = \frac 1 {kn}  \sum_{x,y \in X}  {N_{f,k}(x,y)}   \log_2 (\frac {k} {N_{f,k}(x,y)} ).
$$
\end{defn}
Thus 
$$0 \leq \entfk = H^*(p_{f,k}) - \log_2 n \leq  \log_2 n.$$

We usually leave out the ``shifted'' in the following, although $H$ is not defined as an entropy.
We start with three examples.
\begin{example}
If $f$ is the identity function on $X$, then
\begin{align*}
N_{f,k} (x,y) & = \left\{\begin{array}{ll}
k & \text{if } x=y, \\ 0 & \text{otherwise},
\end{array}
\right. \\
\entfk & = 0.
\end{align*}
\end{example}

\begin{example}
If $f$ is a cyclic permutation, then $N_{f,n}(x,y) = 1$ for all $x,y\in X$,
$p_{f,n}(x,y) = 1/n^2$ is the uniform distribution on $X^2$ with
Shannon entropy $ 2 \log_2 n$, and
\begin{equation*}
\label{maxEnt}
H_{f,n} =  \log_2 n.
\end{equation*}
For a positive integer $m$, we have for all $x,y\in X$
\begin{equation*}
\label{Nml}
\begin{aligned}
N_{f,nm}(x,y) & = m, \\
H_{f,nm}  & =  \log_2 n.
\end{aligned}
\end{equation*}
\end{example}
\begin{example}
\label{linCongGen}
We take $X$ to be a field with $n$ elements, $a, b \in X$ with $a(a-1) \neq 0$, and consider
the \emph{linear congruential generator} $f$ given by $f(x) = ax +b$.
Then
$$
f^{(j)} (x) = a^j x + \frac {(a^j -1) b} {a-1}
$$
for all $j \geq 0$.
Furthermore, let $\ell$ be the order of $a$ in the multiplicative group of $X$.
Then the functional graph of $f$ consists of one cycle $C_0 = \{ x_0 \}$ with the fixed point
$x_0 = -b/(a-1)$ and length $c_0=1$, plus $(n-1)/\ell$ cycles $C_1, \ldots, C_{(n-1)/\ell}$
of length $\ell$.
For a positive integer $m$, $k = \ell m$,  $1 \leq i \leq (n-1)/\ell$ and $x,y \in C_i$, we have
$N_{f,k}(x,y)  = m$,
and also $N_{f,k}(x_0,x_0) = k$. 
Thus
\begin{align*}
\entfk & = \frac {1} {kn} \bigl(k \log_2 1 + \sum_{1 \leq i \leq (n-1)/\ell} \; \sum_{x,y \in C_i} m \log_2 ( \frac k {m}) \bigr) \\
& = \frac 1 {\ell m \cdot n} \frac{(n-1) \ell^2 } \ell {\,  m \log_2 \ell}  
 =  (1-\frac 1 n) \log_2 \ell.
\end{align*}
\end{example}

\section{Combining functions}
\label{combinSection}
Given functions $f_i \colon X_i \rightarrow X_i$ on pairwise disjoint sets $X_1,\ldots,X_s$,
we can combine them into a function $f \colon X \rightarrow X$ on their union $X = \bigcup_{1\leq j \leq s} X_i$
by setting $f(x) = f_i(x)$ for $x \in X_i$.
The functional graph of $f$ is the disjoint union of those of the $f_i$;
the same holds for the usual notion of graph as the set of pairs $(x,f(x))$.
We write $n_i = \# X_i$ and $n = \sum_{1 \leq i \leq s} n_i = \# X$.
The iteration entropy of $f$ turns out to be a convex linear combination of those of the $f_i$.

\begin{thm}
\label{combinThm}
For a positive integer $k$, we have
\begin{align*}
\entfk &  = \sum_{1 \leq i \leq s} \frac {n_i} n  H_{f_i,k} . 
\end{align*}
\end{thm}

\begin{proof}
For $x,y \in X$, we have:
$$
N_{f,k} (x,y)  = \left\{\begin{array}{ll}
N_{f_i,k} & \text{if } x,y \in X_i \text{ for some } i , \\ 0 & \text{otherwise}.
\end{array}
\right. \
$$
Thus
\begin{align*}
\entfk = & \frac 1 {kn} \sum_{x,y \in X}   N_{f,k}(x,y) \log_2 \frac k {N_{f,k}(x,y)} \\
= & \frac 1 {kn}  \sum_{1 \leq i \leq s}  \sum_{x,y \in X_i}  N_{f_i,k}(x,y) \log_2 \frac k {N_{f_i,k}(x,y)} 
 = \sum_{1 \leq i \leq s} \frac {n_i} n  H_{f_i,k} .
\end{align*}
\end{proof}

\section{The asymptotic iteration entropy}
\label{estimating}
The \emph{functional graph} of an arbitrary function $f\colon X \rightarrow X$
has $X$ as its set of nodes and a directed edge from $x$ to $y$ if $f(x)=y$.
The underlying undirected graph consists of undirected connected components $T_i$
each containing a cycle $C_i$, for various values of $i$.
We consider these subgraphs as subsets of $X$, ignoring the order imposed by applications of $f$.
The nodes in $T_i \setminus C_i$ form various \emph{preperiod trees}.
The subgraph $T_i$ consists of $C_i$ and all nodes in the preperiod trees attached to $C_i$,
and we let $t_i$ and $c_i$ be the sizes of $T_i$ and $C_i$, respectively.
Figure \ref{squaringmap} gives two explicit examples.

\begin{figure}[!ht]
  \psset{xunit=1cm,yunit=1cm,runit=1cm}%
  $\vcenter{ \hbox{
      \begin{pspicture}(2.5,0)(11,17)

        \rput(8,14){
          \begin{pspicture}(-1,-1)(0,7) \begin{psmatrix}[mnode=circle,
              colsep=0cm, rowsep=0.3cm]
              & & & & & & & \makebox[0.3cm]{1} & & & [mnode=none]\rput(0,0){$\mathbb{F}_{17}^{\times}$} & & & & [mnode=none] \\
              & & & & & & & \makebox[0.3cm]{-1} & & & & & & & [mnode=none] \\
              & & & \makebox[0.3cm]{4} & & & & & & & & \makebox[0.3cm]{-4} & & & [mnode=none] \\
              & \makebox[0.3cm]{2} & & & & \makebox[0.3cm]{-2} & & & & \makebox[0.3cm]{8} & & & & \makebox[0.3cm]{-8} & [mnode=none] \\
              \makebox[0.3cm]{6} & & \makebox[0.3cm]{-6} & & \makebox[0.3cm]{7} & & \makebox[0.3cm]{-7} & & \makebox[0.3cm]{5} & & \makebox[0.3cm]{-5} & & \makebox[0.3cm]{3} & & \makebox[0.3cm]{-3}
              \psset{arrows=->}
              \ncline{5,1}{4,2} \ncline{5,3}{4,2}
              \ncline{5,5}{4,6} \ncline{5,7}{4,6}
              \ncline{5,9}{4,10} \ncline{5,11}{4,10}
              \ncline{5,13}{4,14} \ncline{5,15}{4,14}
              \ncline{4,2}{3,4} \ncline{4,6}{3,4}
              \ncline{4,10}{3,12} \ncline{4,14}{3,12}
              \ncline{3,4}{2,8} \ncline{3,12}{2,8}
              \ncline{2,8}{1,8}
              \nccircle[angleA=0]{->}{1,8}{.3cm}
            \end{psmatrix}
          \end{pspicture}
        }

        \rput(8.8,7){
          \begin{pspicture}(-0.3,-1)(0,5)
            \begin{psmatrix}[mnode=circle, colsep=0.4cm, rowsep=0.4cm]
              & \makebox[0.3cm]{3} & & & \makebox[0.3cm]{-9} &[mnode=none] & \makebox[0.3cm]{8} & \makebox[0.3cm]{-1} \\
              & & \makebox[0.3cm]{9} & \makebox[0.3cm]{5} & &[mnode=none]\rput(0,0){$\mathbb{F}_{19}^{\times}$} & \makebox[0.3cm]{-7} & \makebox[0.3cm]{1} \\
              \makebox[0.3cm]{-4} & \makebox[0.3cm]{-3} & & & \makebox[0.3cm]{6} & \makebox[0.3cm]{-5} & \makebox[0.3cm]{-8} & [mnode=none]\\
              & & \makebox[0.3cm]{4} & \makebox[0.3cm]{-2} & &[mnode=none] & \makebox[0.3cm]{7} & [mnode=none]\\
              & \makebox[0.3cm]{2} & & & \makebox[0.3cm]{-6} & [mnode=none] & & [mnode=none]
              \psset{arrows=->} \ncline{1,2}{2,3} \ncline{1,5}{2,4}
              \ncline{2,3}{2,4} \ncline{2,4}{3,5} \ncline{3,1}{3,2}
              \ncline{3,2}{2,3} \ncline{3,6}{3,5} \ncline{3,5}{4,4}
              \ncline{4,3}{3,2} \ncline{4,4}{4,3} \ncline{5,2}{4,3}
              \ncline{5,5}{4,4}
              \ncline{1,7}{2,7} \ncarc{->}{2,7}{3,7}
              \ncarc{->}{3,7}{2,7} \ncline{4,7}{3,7}
              \ncline{1,8}{2,8} \nccircle[angleA=180]{->}{2,8}{.3cm}
            \end{psmatrix}
          \end{pspicture}
        }
%
      \end{pspicture}
    }}$
   \caption{The function $x \mapsto x^2$ on the units modulo $17$ and $19$.}
  \label{squaringmap}
\end{figure}
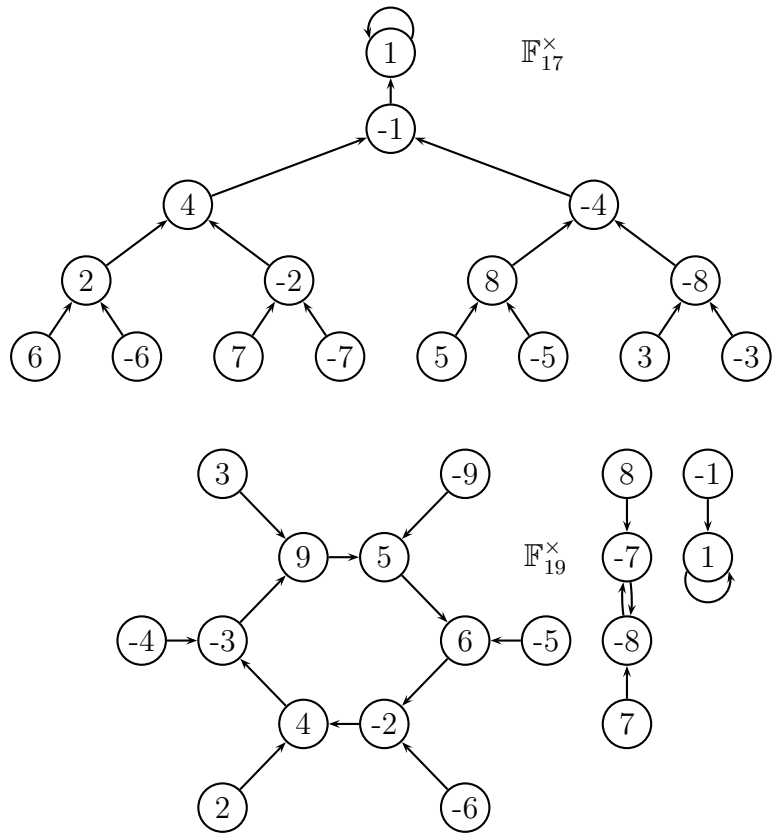

\begin{defn}
\label{asymEnt}
Let $f$ be a function with a functional graph of component sizes $t_i$
and cycle sizes $c_i$ for $1 \leq i \leq s$, as above.
Then
\begin{equation}
\label{Hstar}
H_{f,\infty} = 
 \frac 1 n  \sum_{1\leq i \leq s} t_i \log_2 c_i
\end{equation}
is the \emph{asymptotic (shifted) iteration entropy} of $f$.
\end{defn}
If $f_i$ denotes the restriction of $f$ to $T_i$, operating on $t_i$ values,
then
\begin{align}
H_{f_i,\infty} = & \log_2 c_i, \nonumber \\
H_{f,\infty} = &
  \sum_{1\leq i \leq s}  \frac {t_i} n H_{f_i,\infty},
\label{entinfcombin}
\end{align}
similar to Theorem \ref{combinThm}.

\begin{thm}
\label{generalBound}
For $k \geq 4n \geq 77$,
the following hold.
\renewcommand{\theenumi}{(\roman{enumi})}
\begin{enumerate}
\item
\begin{equation}
\label{generalHbound}
| \entfk - H_{f,\infty} | \leq \frac {4n  \log_2   k} k .
\end{equation}
\item
If $f$ is a permutation, then $c_i = t_i$ for all $i$ and
\begin{align}
| \entfk - \entinf | & \leq \frac {3 n  \log_2 n} k \nonumber \\
\entinf + H^*(\frac {c_1} n ,\ldots, \frac {c_s} n) &  =  \log_2 n ,
 \label{HforPerms} 
\end{align}
Here $H^*({c_1}/ n ,\ldots, {c_s}/ n)$ is the Shannon entropy of the distribution on $s$ elements (the cycles)
with probabilities ${c_1}/ n ,\ldots, {c_s} / n$.
If $f$ is cyclic, then $\entinf =  \log_2 n$.
If $f$ is the identity function, then $\entinf = 0$.
\item
If $f$ is a permutation and $k$ an integer multiple of the order
$\mathrm{lcm} (c_1, \ldots, c_s)$ of $f$, then
\begin{equation*}
\entfk = \entinf .
\end{equation*}
\item
For any $f$, we have
$0 \leq \entinf \leq  \log_2 n$,
and $\entinf = \log_2 n$ if and only if $f$ is a cyclic permutation.
\end{enumerate}
\end{thm}

\begin{proof}
(i)
We start with a single connected component $X$ containing a single cycle $C \subseteq X$ of size $c$. 
The depth $d$ of the functional graph on $X$ is the maximal number of edges on a directed path within it
that terminates in its first point on the cycle;
this equals the maximal number of nodes on such paths minus 1.
Cyclic points do not contribute to this depth.
In Figure \ref{squaringmap}, we have $d=4$ in the graph at the top,
and $d=1$ at the bottom.
We consider the division with remainder
\begin{equation}
\label{divRem}
k= mc +r,
\end{equation}
 with $0 \leq r <c$.
The iterations of $f$ up to $f^{k-1}$ send each initial value on a cycle $m$ times around the cycle, 
and then up to $r$ steps further.
Thus if $x$ and $y$ are on the same cycle, then the orbit of $x$ includes $y$ $m$ times,
plus possibly one more time,
namely if the distance (in the directed functional graph) from $x$ to $y$ is less than $r$.
An off-cycle value spends at most $d$ steps before reaching its root on the cycle,
and then cycles around for at least $k-d$ steps.
Thus for $x,y \in X$, there is an integer $u(x,y)$ so that
\begin{align}
N_{f,k} (x,y) = & \left\{\begin{array}{ll}
 m  + u(x,y) \text{ with } - \lceil \frac d {c} \rceil \leq u(x,y) \leq 1 &
 \text{if } y \in C, \\ 
 u(x,y) \text{ with } 0 \leq u(x,y) \leq 1 & \text{if } y \in X \setminus C,\\
 0 & \text{otherwise},
\end{array}
\right.
\label{numbersCycle}
 \\
\entfk  = & \frac 1 {kn} \sum_{{x\in X} \atop {y \in C} } ( {m + u(x,y)} )
\log_2 \frac { k} {m+ u(x,y)}
\label{contrXC} \\
& + \frac 1 {kn} \sum_{{x\in X} \atop { {y \in X \setminus C} \atop {u(x,y)=1} }}{u(x,y)}
\log_2  \frac  {k} {u(x,y)} ,
\label{hiTwo} 
\end{align}

Since $n \geq c$ and $k\geq 4n > 4d$, we have $m - \lceil d/c \rceil > 0$.
We write
\begin{equation}
\label{infXC}
\entinf =   \log_2  {c} 
=  \sum_{{x\in X} \atop {y \in C} } \frac {\log_2 c}{cn}.
\end{equation}
%

For the error bound, we first bound the difference of the contributions of $(x,y) \in X \times C$ to
(\ref{contrXC}) and (\ref{infXC}).
This proceeds in two steps, first ignoring the logarithmic factors.
We use
\begin{align}
\delta(x,y) & = \frac {m + u(x,y)} {kn} - \frac 1 {c n} =  \frac{cm + c u(x,y) - k} {c kn}
=  \frac{c u(x,y) - r} {c kn}, \label{deltaDef}\\
\delta(x,y) & \leq  \frac{c - r} {c kn} \leq \frac c {ckn}, \nonumber\\
\delta(x,y) & \geq \frac {-c (d/c+1) - r} {c kn} > \frac {-d - c -r} {c kn}, \nonumber \\
|\delta(x,y)| & \leq \frac {d + c + r} {c kn} \leq \frac 2 {c k},
\label{deltaEstimateTwo}
\end{align}
  since $r < c \leq d+c \leq n$.
For the logarithms we consider, again for $x\in X$ and $y \in C$,
$$
\epsilon(x,y) = \frac {\delta(x,y)} {1/c n} = c n \delta(x,y),
$$
so that 
\begin{align}
|\epsilon(x,y)| & \leq \frac {2n} k \leq \frac 1 2, \nonumber \\
| \log_2 (1+ \epsilon (x,y)) | & \leq |2 \epsilon(x,y) | \leq \frac {4n} k,
\label{epsEstimateTwo} \\
\frac {m + u(x,y)} {kn} 
 & = \frac 1 {c n} \cdot (1+\epsilon(x,y)) .  \nonumber
\end{align}

The difference between the contributions of  $(x,y) \in X \times C$ to $\entfk$ and to
$\entinf $ is
\begin{align}
\alpha(x,y)   & = 
\frac {m + u(x,y)} {kn}
\log_2 \frac { k} {m+ u(x,y)}
 -  \frac {\log_2 c } {c n}
\nonumber \\
& = (\frac 1 {c n} + \delta(x,y) ) \bigl(\log_2 c  - \log_2(1+\epsilon(x,y) ) \bigr)
- \frac {\log_2 c } {c n}  \label{alphaDef}  \\
& = - \delta(x,y) \log_2 c  - \frac 1 {c n} \log_2(1+\epsilon(x,y) ) 
\nonumber \\
& \quad  -  \delta(x,y)  \log_2(1+\epsilon(x,y) ) . \nonumber
\end{align}

From (\ref{deltaEstimateTwo}) and (\ref{epsEstimateTwo}), we have, as in a Cauchy-Schwartz inequality,
\begin{equation*}
\label{alphaEstimate}
\begin{aligned}
| \alpha(x,y) | & \leq \frac {2 \log_2 c} {c k} + \frac 1 {c n} \cdot \frac {4n} k 
+ \frac {2} {c k} \cdot \frac {4n} k \\
&  = \frac {1} {ck} (2 \log_2  c +4 + \frac {8n} k )
  \leq \frac {3 \log_2 n} {c k}.
  \end{aligned}
  \end{equation*}
%

In total, we find
\begin{align}
|\entfk - \entinf| & \leq  \sum_{{x\in X} \atop {y \in C} } | \alpha(x,y) | +
\bigl | \sum_{{x\in X} \atop { {y \in X \setminus C} \atop {u(x,y)=1} }} \frac {u(x,y)} {kn} 
\log_2  \frac  {k} {u(x,y)} \bigr| 
\label{HminusHstarIneq1} \\
& \leq  \sum_{{x\in X} } \bigl( \sum_{{y\in C} } \frac {3 \log_2 n} {c k}
+  \sum_{y \in X}   \frac { \log_2 k} {kn} \bigr)
\label{HminusHstarIneq2} \\
& \leq    \sum_{{x\in X} } \bigl( \frac {3 \log_2 n} k + \frac {\log_2 k} {k} \bigr) 
\label{HminusHstarIneq3}\\
& \leq \frac {3n \log_2 n + n \log_2 k} k 
 \leq \frac {4n  \log_2  k} k.
 \label{HminusHstarIneq4}
\end{align}

We now turn to the general case, with connected components $T_i$ of size $t_i$
containing a cycle $C_i$ of size $c_i$,
and let $f_i $
be the restriction of $f$ to $T_i$, for $1 \leq i \leq s$.
Thus $\sum_{1 \leq i \leq s} t_i = n$, the graph of each $f_i$ contains just one component $T_i$, and
 $f$ is the combination of all $f_i$ in the sense of Section \ref{combinSection}.
From Theorem \ref{combinThm} and (\ref{entinfcombin}),
we have
\begin{equation*}
\entfk -\entinf 
 =  \sum_{1 \leq i \leq s}  \frac {t_i} n  (H_{f_i,k}  -  H_{f_i,\infty}) .
\end{equation*}

Since for the single-cycle function $f_i$, $t_i$ plays the role of $ n$ in (\ref{Hstar})
and 
$H_{f_i,\infty} = \log_2 c_i$, 
it follows from (\ref{generalHbound}) that
\begin{equation}
\label{HminusHstarIneq5}
| \entfk -\entinf | \leq \sum_{1 \leq i \leq s}  \frac {t_i} n | H_{f_i,k} - H_{f_i,\infty} | 
\leq  \sum_{1 \leq i \leq s}  \frac {t_i} n \cdot \frac {4 t_i \log_2 k} k 
\leq  \frac {4n \log_2 k} k.
\end{equation}

(ii)  As in the proof of (i), we first assume $f$ to be a cyclic permutation.
Then $d=0$ and $0 \leq u(x,y) \leq 1$ in the first line of (\ref{numbersCycle}), and
$u(x,y)=0$ in the second line, since $X=T=C$.
In the equation (\ref{hiTwo}), the last summand vanishes in (\ref{HminusHstarIneq1})
through (\ref{HminusHstarIneq3}), and the bound in (\ref{HminusHstarIneq4}) becomes
$( {3n  \log_2 n})/k$.

Representing a general permutation as a combination of cyclic ones
gives this bound also in (\ref{HminusHstarIneq5}). Furthermore, we have
\begin{equation*}
\entinf + H^*(\frac {c_1} n ,\ldots, \frac {c_s} n)
 = \frac 1 n  \sum_{1 \leq i \leq s} {c_i}  \log_2  {c_i }  + \sum_{1 \leq i \leq s} \frac {c_i} n \log_2 \frac n {c_i}
 =  \log_2 n.
\end{equation*}

(iii) In addition to the properties in (ii), now $r=0$ in (\ref{divRem}) and $u(x,y) = 0$
in the first line of (\ref{numbersCycle}).
Therefore $\delta(x,y) = 0$ in (\ref{deltaDef})
and $\alpha(x,y) = 0$ in (\ref{alphaDef}).

(iv) 
Using (\ref{HforPerms}) and $c_i \leq t_i$ for all $i$, we have
$$
\entinf =  \frac 1 n  \sum_{1\leq i \leq s} t_i \log_2 c_i \leq  \frac 1 n  \sum_{1\leq i \leq s} t_i \log_2 t_i
 =  \log_2 n - H^*(\frac {t_1} n ,\ldots, \frac {t_s} n) \leq  \log_2 n.
$$
The first inequality is strict unless $t_i = c_i$ for all $i$, and
$H^*({t_1} / n ,\ldots, {t_s}/ n) = 0$
if and only if $s=1$ and thus $t_1 = n$.
Hence $\entinf =  \log_2 n$ if and only if $f$ is a cyclic permutation.
\end{proof}

The main term $\entinf$ in Theorem \ref{generalBound} (i) is independent of $k$,
and the error bound goes to zero with growing $k$.
\cite{sin66} calls the expression $ \sum_{1\leq i \leq s} t_i \log_2 c_i$
a \emph{cross-entropy}, but does not discuss it further.
It plays a role in modern cryptanalysis of classical ciphers,
as in \cite{las17}.
We are not aware of other sources for this cross-entropy.

While Definition \ref{defItEnt} of the shifted iteration entropy is stated as  a sum over $n^2$ terms,
the number of summands in the asymptotic shifted iteration entropy is only the number of cycles.
Of course, the different cycles lengths seem, in general, hard to compute.

If the functional graph of a function $f$ on $n$ elements
 contains a connected component
of  size $t_1 = \tau n$ with a cycle of length $c_1 = n^\gamma$, plus possibly other components,
then
\begin{equation}
\label{longCycle}
\entinf \geq \tau  \gamma \log_2 n.
\end{equation}
If $f$ is, in addition, a permutation, then
\begin{equation}
\label{longCyclePerm}
\entinf \geq  \tau  \log_2 n  + \tau \log_2 \tau.
\end{equation}

\section{Tree surgery}
How do the asymptotic iteration entropies of two distinct but closely related functions compare?
We discuss three ways of slightly modifying a functional graph and their effect on the asymptotic iteration entropy.

Suppose we remove one ``leaf'' (most outlying node) from one of the preperiod trees.
Thus we consider components and cycles $T_i \supseteq C_i$, and set
 $t'_i = t_i$ and $c'_i = c_i$ for all $i$, except that $t'_s = t_s -1$, assuming $t_s > c_s$.
 We take a new function $f'$ on a set with $n-1$ elements whose graph has these parameters.
Then
\begin{align*}
\Delta & = \entinf - H_{f',\infty} = (\frac 1 n - \frac 1 {n-1} ) \sum_{1\leq i < s} t_i \log_2 c_i
+ (\frac  {t_s} n - \frac {t_s -1} {n-1} ) \log_2 c_s \\
& = \frac 1 {n (n-1)} (( \sum_{1\leq i < s} t_i \log_2 c_i) - (n-t_s) \log_2 c_s )
= \frac 1 {n (n-1)}  \sum_{1\leq i < s} t_i \log_2 \frac {c_i} {c_s} .
\end{align*}
If $s=1$, then $\Delta = 0$, and if $s\geq 2$ and $C_s$ is a smallest cycle, then $\Delta \geq 0$.

An alternative is to enlarge $C_s$ at the expense of $T_s$,
by moving one node in $T_s$, at distance 1 from $C_s$, into $C_s$.
Thus $c'_i = c_i$ and $t'_i = t_i$ for all $i$, except that $c'_s = c_s +1$ and $t'_s = t_s -1$, assuming $t_s > c_s$.

We take a new function $f'$ on $X$ whose graph has these parameters.
Then
\begin{align*}
\Delta & = \entinf - H_{f',\infty} = 
 \frac 1 n ( {t_s \log_2 c_s}  -  {(t_s -1) \log_2(c_s+1)} ) \\
& = \frac 1 {n } (t_s \log_2 \frac {c_s}  {c_s+1} +  \log_2 (c_s +1) )
\end{align*}
Now $\Delta$ may be positive, negative, or zero.
If we replace $\log_2 ( 1 - \frac 1  {c_s+1})$ by $-1/(c_s+1)$, then the value is positive
if and only if $t_s < (c_s+1) \log_2 (c_s+1)$.

For a more general result, we can at least compare two functions one of which is obtained from
the other one by amalgamating components and cycles.
We take four sequences of positive integers representing component and cycle sizes:
\begin{equation*}
\label{amalgamate}
\begin{aligned}
t & = (t_1, \ldots, t_s), \\
c & = (c_1, \ldots, c_s), \\
t ' &= (t'_1, \ldots, t'_r), \\
c'  & = (c'_1, \ldots, c'_r),
\end{aligned}
\end{equation*}
with $r < s$, $n = \sum_{1 \leq i \leq s} t_i =  \sum_{1 \leq j \leq r} t'_j$,
and $c_i \leq t_i$ and  $c'_i \leq t'_i$ for all $i$.
We say that $(t,c) \prec (t',c')$ if there exist pairwise disjoint sets $S_1, \ldots, S_r \subseteq \{1,\ldots,s\}$
such that $t'_j = \sum_{i\in S_j} t_i$ 
and $c'_j = \sum_{i\in S_j} c_i$ for $1 \leq j \leq r$.
For example, if $r=s-1$, $S_j = \{j\}$ for $j<r$, and $S_r = \{s-1,s\}$, then
we may imagine the corresponding cycles $C_{s-1}$ and $C_s$ cut open at one point and then joined to form one cycle,
with all preperiod trees remaining attached.

\begin{thm}
\label{comparison}
Let $f$ and $f'$ be functions on a set of $n$ elements whose functional graphs have component and cycle sizes $t$ and $c$
and $t'$ and $c'$, respectively.
If $(t,c) \prec (t',c')$, then $\entinf < H_{f',\infty}$.
\end{thm}
\begin{proof}
Inductively, it is sufficient to consider the example above with
 $r=s-1$, $S_j = \{j\}$ for $j \leq s-2$ and $S_{r} = \{s-1,s\}$.
Thus $c'_j = c_j$ and $t'_j = t_j$ for $j < r$, and
$t'_{r} = t_{s-1} + t_s$ and $c'_{r} = c_{s-1} + c_s$.
Then
\begin{align*}
n ( \entinf- H_{f',\infty} ) & =\sum_{1\leq i \leq s} {t_i}  \log_2  {c_i } \\
& \quad - \bigl(  \sum_{1\leq i \leq s-2} {t_i}  \log_2  {c_i} + ({t_{s-1} + t_s} ) \log_2   {(c_{s-1} + c_s)} \\
 & = {t_{s-1}} (\log_2 c_{s-1} - \log_2   {(c_{s-1} + c_s)}) \\
& \quad + {t_{s}} (\log_2 c_{s} - \log_2   {(c_{s-1} + c_s)}) \bigr) < 0.
\end{align*}
\end{proof}
In other words, amalgamating components as above
increases the asymptotic iteration entropy.

\section{Examples}
We present some examples.
\begin{example}
\label{powermap}
The \emph{power map} $x \mapsto x^e$ in a finite field or a ring $\Z/N\Z$, for fixed $e, N \geq 2$,
is of cryptographic interest. Its iterations include the \emph{power generator} for pseudorandom sequences
and, with $e=2$, the Blum-Blum-Shub and Hofheinz-Kiltz-Shoup cryptosystems
(\cite{blublu86,hofkil13}).
\cite{fripom01} exhibit lower bounds on the order (or period) of this function, that is,
the lcm of all cycle lengths.
\cite{kurpom05} show that the maximal value (over all initial points) equals the order of $e$ modulo $M$,
where $M$ is the largest divisor of the Carmichael value $\lambda(N)$ that is coprime to $e$.
They prove a lower bound of about $N^{1/2}$ for a ``Blum integer'',
which is the product of two primes $p$ and $q$ for which $p-1$ and $q-1$
have a large prime divisor.
\cite{shahu11} show a similar result in finite fields, and \cite{sha11} for the case where $N$ is a prime power.
\cite{pomshp17} prove several results about the number of cycles in the functional graph,
among them a lower bound of $p^{5/12 + o(1)}$ for infinitely many primes $p$.
\cite{choshp17} compute the number of cyclic points, the average cycle length, and other quantities
for such maps, extending the work of \cite{vassha04} on $e=2$ (who use the Extended Riemann Hypothesis).
Corollary \ref{longCycle} gives a lower bound on the
 asymptotic iteration entropy of a permutation whose graph has a large cycle.

\end{example}
\begin{example}
Let $n$ be a power of 2 and suppose that the functional graph of $f$ contains a complete binary tree
whose root is mapped under $f$ to the only cycle in the graph, consisting of one point.
Figure \ref{squaringmap} illustrates this with the squaring function $x \mapsto x^2$ 
 for the Fermat prime $p=17$ on the unit group $X = \F_{17}^{\times}$ with $n=16$ elements.
Thus $s=1$, $T_1 = X$,
$t_1 = n$, $c_1 = 1$, and $\entinf = 0$. Theorem \ref{generalBound} says that
$\entfk \leq (4n  \log_2  k) / k$; the latter value tends to zero with growing $k$.
Under our measure, this function exhibits ``small'' iteration entropy.
\end{example}

\begin{example}
Suppose that $n$ is even and $f$ has 
one cycle $C_1$ of
size $n/2$, with a one-node tree attached to each point on the cycle.
Thus $s=1$, $c_1 =n/2$, $t_1 = n$.
The \emph{benzene ring} on $\F_{19}^\times$ at lower left in Figure \ref{squaringmap} is an example
on $n=12$ points.
Then
$$
\entinf = \frac {t_1} n \log_2 \frac n 2  = ( \log_2 n ) -1.
$$
If we combine it with a single fixed point $C_2 = \{ x_0 \}$, then the functional graph of
some quadratic function $f$ on a field with $q = n+1$ elements might look like this.
Then 
\begin{equation*}
 \entinf = \frac {q-1} q \log_2 ((q-1) /2) + \frac 1 q \log_2 1
 \approx (1- \frac 1 q) \cdot (\log_2 q -1).
\end{equation*}
Under our measure, both are ``large'' asymptotic iteration entropies.
\end{example}

\begin{example}
In \cite{bopgat17}, the \emph{ElGamal function} $f\colon x \mapsto g^x$ on $\F_p$ is studied,
where $p$ is a prime number and $g$ is a generator of the multiplicative group of $\F_p$.
This function occurs in some cryptographic protocols.
For the two primes $1009$ and $10\, 009$, there are $288$ and $3312$ generators,
respectively. Figures \ref{1009ent} and \ref{10009ent} show the asymptotic iteration entropies
of all these functions, normalized as $H_{f,\infty} / \log_2 p$, so that all values lie between 0 and 1.
\begin{figure}
\begin{center}
 \includegraphics[width=1.1\linewidth]{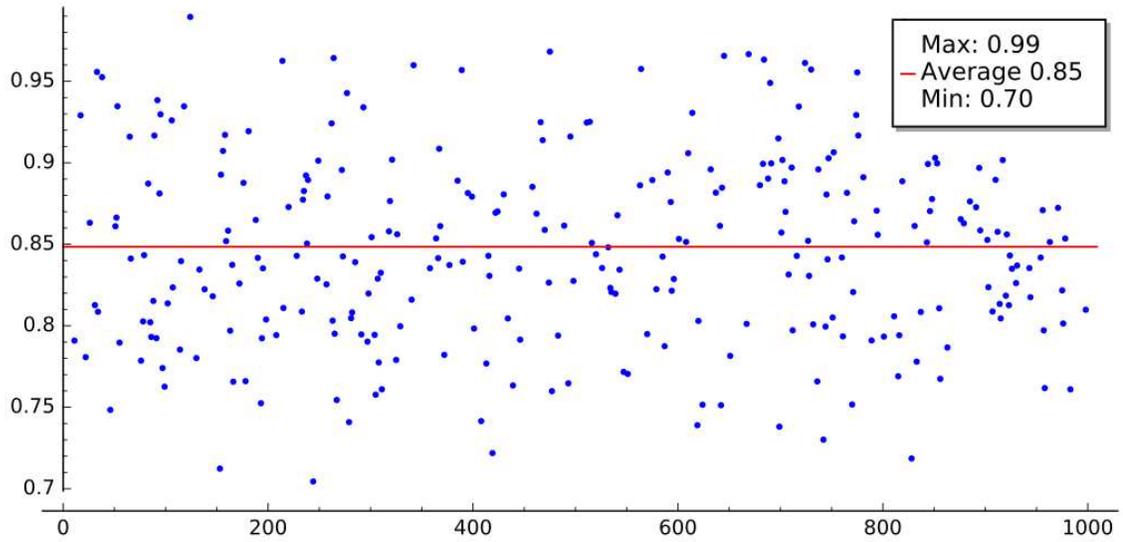}
 \end{center}
\caption{The asymptotic iteration entropies for the 288 ElGamal functions on $\F_{1009}$.
The red line at 0.85 indicates the average.}
\label{1009ent}
\end{figure}
\begin{figure}
\begin{center}
 \includegraphics[width=1.1\linewidth]{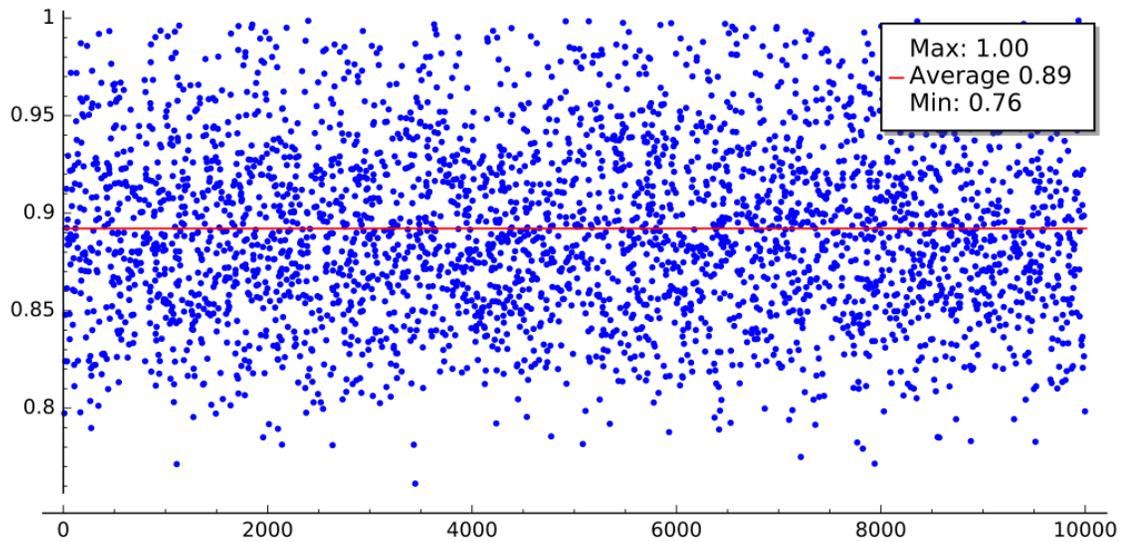}
 \end{center}
\caption{The asymptotic iteration entropies for the 3312 ElGamal functions on $\F_{10009}$.
The red line near 0.89 indicates the average.}
\label{10009ent}
\end{figure}
The values lie, on average, about 10 to 15\%\ below the maximal value of 1.
The variances were also computed but are too small to be shown.
\end{example}

\section{Open questions}
\begin{itemize}
\item
What is the average asymptotic iteration entropy of a random permutation? Or a random function?
(\ref{longCyclePerm}) provides a lower bound for individual permutations.
Is this, with the proper value of $\tau \approx 0.62433$, also a (lower or upper) bound on the average?
What is the average value of $\max\{ t \log_2 c\}$ for random functions, where $t$ runs through the component sizes
and $c$ is the size of the component's cycle?
The joint distribution of $(t,c)$ does not seem to have been studied.
The average size $\mu n$
with $\mu \approx 0.75788$ (\cite{flaodl90}) of the giant component might be a lower bound,
except that components with a fixed point ($c=1$) would have to be ruled out.
\item
The linear congruential generator of Example \ref{linCongGen} is well-known to be insecure
(see \cite{boy89}) and hence does not provide (pseudo)random values by iteration.
For large $\ell$, say $\ell = n-1$, its asymptotic iteration entropy is close to the maximal value
of $ \log_2 n$.
Thus large iteration entropy does not imply (pseudo)randomness.
Is the converse true in some sense?
\item
What is the relation of the (asymptotic) iteration entropy to usual notions of random generation?
A function on a finite set contains only a finite amount of information (or Shannon entropy)
and its iterates, from a uniformly random starting value,
 do not generate a statistically random sequence of elements.
But one may ask for a modest amount of equidistribution
(see  \cite{bopgat17} for the ElGamal function)
or whether some form of pseudorandomness is obtainable.
Conversely, does pseudorandomness imply that the function
is a permutation?
For example, the squaring function on the set of quadratic residues with Jacobi symbol 1
modulo a special type of RSA modulus is used in \cite{hofkil13};
it is pseudorandom under the assumption that such moduli are hard to factor,
it is a permutation, and in general not cyclic.
\end{itemize}

\section{Acknowledgements}
The images for Figures \ref{1009ent} and  \ref{10009ent} were produced by Lucas Perin,
whose help is much appreciated.
Many thanks go to Alina Ostafe and to Igor Shparlinski for valuable suggestions, corrections,
and pointers to the literature.

\bibliographystyle{plainnat}
\bibliography{IterationEntropy}

\end{document}